\RequirePackage{silence}
\documentclass[10pt,english]{article}

\usepackage[margin= 2.2 cm, bottom=19.5mm,footskip=7mm,top=20mm]{geometry}

\usepackage{amsthm}
\usepackage{amsmath}
\usepackage{amssymb}
\usepackage{setspace}
\usepackage{mathtools}
\usepackage{verbatim}
\usepackage{csquotes}
\usepackage{graphicx}
\usepackage[hidelinks]{hyperref}
\usepackage{thm-restate}
\usepackage{cleveref}
\usepackage{graphicx}
\usepackage{enumitem}
\usepackage{framed}
\usepackage{subcaption}

\usepackage{floatrow}
\usepackage[T1]{fontenc}
\usepackage{tikz}
\usepackage{mathdots}
\usepackage{xcolor}
\usepackage{diagbox}
\usepackage{colortbl}
\usepackage[absolute,overlay]{textpos}
\usetikzlibrary{calc}
\usetikzlibrary{decorations.pathreplacing}
\usetikzlibrary{positioning,patterns}
\usetikzlibrary{arrows,shapes,positioning}
\usetikzlibrary{decorations.markings}

\tikzstyle{edge}=[very thick]
\definecolor{bostonuniversityred}{rgb}{0.8, 0.0, 0.0}
\definecolor{arsenic}{rgb}{0.23, 0.27, 0.29}
\tikzstyle{diredge}=[postaction={decorate,decoration={markings,
		mark=at position .95 with {\arrow[scale = 1]{stealth};}}}]

\newcommand{\fitellipsis}[2] 
{\draw [fill=green]let \p1=(#1), \p2=(#2), \n1={atan2(\y2-\y1,\x2-\x1)}, \n2={veclen(\y2-\y1,\x2-\x1)}
    in ($ (\p1)!0.5!(\p2) $) ellipse [ x radius=\n2/2+0cm, y radius=0.1cm, rotate=\n1];
}

\theoremstyle{plain}

\newtheorem*{thm*}{Theorem}
\newtheorem{thm}{Theorem}
\Crefname{thm}{Theorem}{Theorems}

\newtheorem*{lem*}{Lemma}
\newtheorem{lem}[thm]{Lemma}
\Crefname{lem}{Lemma}{Lemmas}

\newtheorem*{claim*}{Claim}
\newtheorem{claim}{Claim}
\crefname{claim}{Claim}{Claims}
\Crefname{claim}{Claim}{Claims}

\newtheorem{prop}[thm]{Proposition}
\Crefname{prop}{Proposition}{Propositions}

\crefname{cor}{Corollary}{Corollaries}

\crefname{conj}{Conjecture}{Conjectures}

\Crefname{qn}{Question}{Questions}

\newtheorem{obs}[thm]{Observation}
\Crefname{obs}{Observation}{Observations}

\newtheorem{ex}[thm]{Example}
\Crefname{ex}{Example}{Examples}

\theoremstyle{definition}

\Crefname{prob}{Problem}{Problems}

\Crefname{defn}{Definition}{Definitions}

\newtheorem*{defn*}{Definition}

\theoremstyle{remark}

\renewenvironment{proof}[1][]{\begin{trivlist}
\item[\hspace{\labelsep}{\bf\noindent Proof#1.\/}] }{\qed\end{trivlist}}

\expandafter\def\expandafter\normalsize\expandafter{%
    \normalsize
    \setlength\abovedisplayskip{8pt}
    \setlength\belowdisplayskip{8pt}
    \setlength\abovedisplayshortskip{4pt}
    \setlength\belowdisplayshortskip{4pt}
}

\usepackage[square,sort,comma,numbers]{natbib}
 \setlist[itemize]{leftmargin=*}

\newcommand{\sun}{f}
\newcommand{\cA}{\mathcal{A}}
\newcommand{\HH}{\mathcal{H}}
\newcommand{\cF}{\mathcal{F}}

\renewcommand{\ex}{\mathrm{ex}}

\renewcommand{\S}[3]{\mathcal{S}^{(#1)}_{#2}(#3)}
\DeclareFontFamily{OT1}{pzc}{}
\DeclareFontShape{OT1}{pzc}{m}{it}{<-> s * [1.10] pzcmi7t}{}
\DeclareMathAlphabet{\mathpzc}{OT1}{pzc}{m}{it}

\title{\vspace{-0.8cm} Tur\'an numbers of sunflowers}

\author{
Domagoj Brada\v{c}\thanks{Department of Mathematics, ETH, Z\"urich, Switzerland. Email: \href{mailto:domagoj.bradac@math.ethz.ch} {\nolinkurl{domagoj.bradac@math.ethz.ch}}.}
\and
Matija Buci\'c\thanks{School of Mathematics, Institute for Advanced Study and Department of Mathematics, Princeton University, Princeton, USA. Email: \href{mailto:matija.bucic@ias.edu} {\nolinkurl{matija.bucic@ias.edu}}.}
 \and
Benny Sudakov\thanks{Department of Mathematics, ETH, Z\"urich, Switzerland. Email:
\href{mailto:benjamin.sudakov@math.ethz.ch} {\nolinkurl{benjamin.sudakov@math.ethz.ch}}.
Research supported in part by SNSF grant 200021\_196965.}
}
 \date{}

\begin{document}

\maketitle

\vspace{-0.5cm}
\begin{abstract}
    A collection of distinct sets is called a \emph{sunflower} if the intersection of any pair of sets equals the common intersection of all the sets. Sunflowers are fundamental objects in extremal set theory with relations and applications to many other areas of mathematics as well as theoretical computer science.  A central problem in the area due to Erd\H{o}s and Rado from 1960 asks for the minimum number of sets of size $r$ needed to guarantee the existence of a sunflower of a given size. Despite a lot of recent attention including a polymath project and some amazing breakthroughs, even the asymptotic answer remains unknown.
    
    We study a related problem first posed by Duke and Erd\H{o}s in 1977 which requires that in addition the intersection size of the desired sunflower be fixed. This question is perhaps even more natural from a graph theoretic perspective since it asks for the Tur\'an number of a hypergraph made by the sunflower consisting of $k$ edges, each of size $r$ and with common intersection of size $t$. For a fixed size of the sunflower $k$, the order of magnitude of the answer has been determined by Frankl and F\"{u}redi. In the 1980's, with certain applications in mind, 
    Chung, Erd\H{o}s and Graham and Chung and Erd\H{o}s considered what happens if one allows $k$, the size of the desired sunflower, to grow with the size of the ground set. In the three uniform case $r=3$ the correct dependence on the size of the sunflower has been determined by Duke and Erd\H{o}s and independently by Frankl and in the four uniform case by Buci\'{c}, Dragani\'{c}, Sudakov and Tran. We resolve this problem for any uniformity, by determining up to a constant factor the $n$-vertex Tur\'an number of a sunflower of arbitrary uniformity $r$, common intersection size $t$ and with the size of the sunflower $k$ allowed to grow with $n$.
\end{abstract} 


\section{Tur\'an numbers of sunflowers}
A family $A_1,\ldots,A_k$ of distinct sets is said to be a {\em sunflower} if there exists a {\em kernel} $C$ contained in each of the $A_i$ such that the {\em petals} $A_i\setminus C$ are disjoint. For $r,k\ge 1$, let $\sun_r(k)$ denote the smallest natural number with the property that any family of $\sun_r(k)$ sets of size $r$ contains an ($r$-uniform) sunflower with $k$ petals. The celebrated Erd\H{o}s-Rado theorem \cite{ER60} from 1960 asserts that $\sun_r(k)$ is finite; in fact Erd\H{o}s and Rado gave the following bounds:
\begin{equation}\label{eq:Erdos-Rado bound}
(k-1)^r \le \sun_r(k) \le (k-1)^r r!+1.
\end{equation}

They conjectured that for a fixed $k$ the upper bound can be improved to $\sun_r(k) \le O(k)^r$. Despite significant efforts, a solution to this conjecture remains elusive. A recent breakthrough in this direction was obtained by Alweiss, Lovett, Wu and Zhang \cite{ALWZ19}. The ideas they developed to tackle this problem have found quite wide applications not only within combinatorics but also in theoretical computer science and probability theory. Their methods have since been pushed a bit further by Rao \cite{Rao20} (Tao \cite{tao} found a different proof of this result using Shannon entropy) and Bell, Chueluecha and Warnke \cite{bell21}, giving the current record of $\sun_r(k) \le O( k \log r )^r$.

About 45 years ago, Duke and Erd\H{o}s \cite{DE77} initiated the systematic investigation of a closely related problem, where one in addition fixes the size of the kernel of the sunflower. In particular, we seek the $r$-uniform sunflower with $k$ petals and kernel of size $t$, which we will denote by $\S{r}{t}{k}$. Duke and Erd\H{o}s asked for the minimum number of sets of size $r$ which guarantees us to be able to find $\S{r}{t}{k}$. 
In other words, they ask for the Tur\'an number\footnote{The Tur\'an number $\ex(n,\HH)$ of an $r$-graph $\HH$ is the maximum number of edges in an $r$-graph on $n$ vertices which does not contain a copy of $\cF$ as a subhypergraph.} of the sunflower $\S{r}{t}{k}$.
Over the years this problem has been reiterated several times \cite{Furedi-survey,chung1997open} including in a recent collaborative ``polymath'' project \cite{Polymath}. 
The case $k=2$ of the problem (a certain fixed intersection size is forbidden) is better known as the restricted intersection problem and has received
considerable attention over the years \cite{FF85, FR87, FW81, KL17, KL20, Sos73}, partly due to its huge impact in discrete geometry \cite{FR90}, communication complexity \cite{Sgall99} and quantum computing \cite{BCW98}.
Another case that has a rich history \cite{Erdos65, EG61, EKR61, Frankl13, Frankl17, Frankl17-newrange, FK18} is $t=0$ (a matching of size $k$ is forbidden); the optimal construction in this case is predicted by the famous Erd\H{o}s Matching Conjecture. 

For fixed $r,t$ and $k$ with $1 \le t \le r-1$ Frankl and F{\"u}redi \cite{FF87} give a conjecture for the correct value of $\ex(n,\S{r}{t}{k})$ up to lower order terms, based on two natural candidates for near-optimal $\S{r}{t}{k}$-free $r$-graphs. They verify their conjecture for $r \ge 2t+3$, but otherwise, with the exception of a few particular small cases, it remains open in general. In terms of asymptotic results, Frankl and F{\"u}redi \cite{FF85} showed in 1985 that  $\ex(n,\S{r}{t}{k})\approx_{r,k} n^{\max\{r-t-1,t\}}$.\footnote{We write $\approx$ here instead of the more commonplace $\Theta$ for clarity. For multi-valued functions we index the $\approx$ with the terms which are considered to be constants.} 

While the above result of Frankl and F{\"u}redi tells us the correct order of magnitude of the Tur\'an number for any sunflower of \emph{fixed} size, as soon as we are interested in finding ``larger'' sunflowers, namely of size even very moderately growing with $n$, their bounds deteriorate drastically. For example, in the graph case ($r=2$) there is an immediate, more precise result, namely $\ex(n,\S{2}{1}{k})\approx nk$ and being able to find such stars, of size growing with $n$, is often very useful in a variety of situations. In general, the problem of obtaining the correct dependence of the Tur\'an number not only on $n$, the size of our underlying graph, but also on $k$, the size of our sunflower, for hypergraphs of higher uniformity turns out to be more challenging. 
The $3$-uniform case is already interesting: Duke and Erd\H{o}s \cite{DE77} and Frankl \cite{Frankl78} showed $\ex(n,\S{3}{1}{k}) \approx nk^2$ and $\ex(n,\S{3}{2}{k})\approx n^2k$, while Chung \cite{chung1983unavoidableS} even managed to determine the answer up to lower order terms. In the 1980's Chung, Erd\H{o}s and Graham \cite{CEG} and Chung and Erd\H{o}s \cite{chung1987unavoidable}, motivated by potential applications to the problem of finding optimal so-called unavoidable graphs and hypergraphs, consider what happens with Tur\'an numbers of such large sunflowers with uniformity higher than $3$. In the four uniform case this was only recently resolved in \cite{unavoidability}. In this paper, we resolve this problem for any uniformity by determining the correct order of magnitude of the Tur\'an number of sunflowers of arbitrary uniformity, in terms of both the size of the underlying graph and the size of the sunflower.

\begin{restatable}{thm}{thmmain}\label{thm:main}
$$
 \ex(n,\S{r}{t}{k})\approx_{r} 
\begin{cases}
n^{r-t-1}k^{t+1} & \text{ if } t \le \frac{r-1}{2}, \\
n^{t}k^{r-t} & \text{ if } t > \frac{r-1}{2}.
\end{cases}
$$
\end{restatable}

The key special case of the above result is when $t=\frac{r-1}{2}$, which we refer to as the \emph{balanced case}, since the petal and kernel of the \emph{balanced sunflower} $\S{2t+1}{t}{k}$ we seek are roughly of the same size. The importance of the balanced case is due to a relatively simple reduction of any other case to the balanced one. Furthermore, it turns out that it is connected to a number of interesting topics ranging from extremal set theory to optimisation theory and theoretical computer science.

The heart of our proof for the upper bound of the balanced case $r=2t+1$ is a reduction to the extremal problem asking for the existence of the following object.
\begin{defn*}
A set family $\mathcal{A} \subseteq \mathcal{P}([N])$ is said to be a $(t+1,t)$-system if: 
\begin{itemize}
    \item $\mathcal{A}$ is intersection closed, i.e. $\forall A,B \in \mathcal{A}$ we also have $A \cap B \in \mathcal{A}$,
    \item any subset of $[N]$ of size $t$ is contained in some set in $\mathcal{A}$ and
    \item for any $A \in \mathcal{A}$ we have $|A| \not\equiv N \pmod {t+1}$.
\end{itemize}
\end{defn*}

The specific case of interest for us is when $N=2t+1$, and in particular our reduction shows that provided such a system does not exist, \Cref{thm:main} holds in the balanced case $t=\frac{r-1}{2}$. 

In general the existence of $(t+1,t)$-systems is related to a number of interesting topics. For example, it played an important role in a recent work on submodular optimisation \cite{submodular}, where using an algebraic argument it was shown $(t+1,t)$-systems do not exist when $t+1$ is a prime power. Through our reduction this implies \Cref{thm:main} when $t+1$ is a prime power. On the other hand, a follow-up work focusing on boolean constraint satisfaction problems \cite{csp} shows that in fact $(t+1,t)$ systems do exist when $t+1$ has two distinct prime divisors. Despite this, and due to the fact that these examples require a relatively large ground set, it turns out that in our case the $(t+1,t)$-systems on ground set of size $N=2t+1$ do not exist for any $t$. Our proof of this statement for all $t$ takes a slightly different approach, it follows closely an argument of Frankl and F\"uredi in \cite{FF85} and relies on the following beautiful extremal set theoretic result of Frankl and Katona from 1979, which was originally inspired by a certain database design problem. They showed that among any collection of $t+1$ subsets of $[t]$ there are $a+1$ of them with intersection of size exactly $a$. 

\textbf{Notation.} Throughout the paper the uniformity of our hypergraphs is treated as a fixed constant in all our use of the asymptotic notation and we always assume that $k,$ the size of the desired sunflower, is at least $2.$

\section{Balanced sunflowers}
In this section we prove the upper bound of \Cref{thm:main} holds in the balanced case. 
\begin{thm}\label{thm:balanced} 
$$
 \ex(n,\S{2t+1}{t}{k})\le O(n^{t}k^{t+1}).
$$
\end{thm}

The proof is split into two parts corresponding to the following two subsections. In the first one we prove a key reduction to the existence problem of the $(t+1,t)$-systems defined in the introduction. In the subsequent one we deal with the existence of $(t+1,t)$-systems.

\subsection{Reduction to the existence of \texorpdfstring{$(t+1,t)$}{(t+1,t)}-systems}
The aim of this subsection is to prove the following key lemma.

\begin{lem} \label{lem:reduction-to-set-system}
    If there exists no $(t+1, t)$-system with $N = 2t+1,$ then $\ex(n,\S{2t+1}{t}{k})\le O(n^{t}k^{t+1}).$
\end{lem}

Before turning to its proof let us define a convenient piece of notation which we will use throughout the paper. Given an $r$-uniform hypergraph $H$ and a subset $S$ consisting of $m$ of its vertices, we define the link graph $L_S$ of $S$ in $H$ to be the $(r-m)$-uniform hypergraph consisting of all $r-m$ sets which together with $S$ make an edge of $H$.

\begin{obs} \label{obs:cover}
Let $H$ be an $r$-uniform hypergraph without $\S{r}{t}{k}$. Given any set $S$ of $t$ vertices in $H$, its link graph $L_S$ in $H$ can be covered using at most $(k-1)(r-t)$ vertices.
\end{obs}
\begin{proof}
    $L_S$ is an $(r-t)$-uniform hypergraph and by assumption that $H$ is $\S{r}{t}{k}$-free, we know $L_S$ has no set of $k$ pairwise vertex disjoint edges. Taking a maximal collection of pairwise disjoint edges in $L_S$ we know it contains at most $(k-1)(r-t)$ vertices and we know that every edge of $L_S$ must contain one of them or we could extend the collection, contradicting maximality. Hence, these vertices provide us with the desired cover.
\end{proof}

Before presenting the formal proof of \Cref{lem:reduction-to-set-system}, we try to give some intuition behind the key ideas. Let $H$ be a $(2t+1)$-uniform hypergraph with no copy of $\S{2t+1}{t}{k}.$ Our goal is to show that $H$ has at most $O(n^t k^{t+1})$ edges.
For each $t$-vertex set $S$ in $H$, we fix a cover $\Phi(S)$ of $L_S$ with $|\Phi(S)| = O(k),$ provided by \Cref{obs:cover}. Suppose we wish to enumerate all the edges of $H.$ One way to do this is to choose $v_1, \dots, v_t$ to be  any $t$ distinct vertices, choose $v_{t+1}$ to be a vertex in $\Phi(\{v_1, \dots, v_t\})$, and choose the remaining vertices $v_{t+2}, \dots, v_{2t+1}$ arbitrarily. As any edge containing $v_1, \dots, v_t$ contains a vertex from $\Phi(\{v_1, \dots, v_t\}),$ we have enumerated all edges. Thus we obtain $|E(H)| = O(n^{2t} k),$ an improvement over the trivial bound $O(n^{2t+1}).$

How can we take this argument further? Suppose we have chosen the vertices $v_1, \dots, v_{t+1}$ as above. If there is an index $i \in [t+1]$ such that for $S = \{v_1, \dots, v_{t+1}\} \setminus \{v_i\},$ the vertex $v_i$ is not in $\Phi(S),$ then we can choose the next vertex $v_{t+2}$ from $\Phi(S)$ and proceed as before. In general, in each step we want to find a set $S$ of $t$ already chosen vertices such that $\Phi(S)$ contains none of the already chosen vertices. If we are always able to find such an $S$ we obtain our desired bound of $O(n^tk^{t+1})$.

More precisely, for each edge $e$ in $H$, we fix some ordering of its vertices and denote $e = (u_1, u_2, \dots, u_{2t+1})$. For an edge $e$ and a set $I \subseteq [2t+1],$ we will write $e_I$ for the subset of the edge consisting of vertices indexed by $I$, in other words $e_I = \{ u_i \, \vert \, i \in I \}.$ Looking at any set of indices $I \in \binom{[2t+1]}{t},$ we know there exists an index $j \in [2t+1] \setminus I$ such that $u_j \in \Phi(e_I).$ This lets us define a function $f_e$ which assigns to each $t$-element subset $I$ of the coordinates a new coordinate $j$, so outside of $I$, such that $u_j$ belongs to the cover $\Phi(e_I)$ of the link graph defined by the vertices of $e$ indexed by $I$. One should think of $f_e$ as the recipe for how to perform our counting strategy above. Since there are only $O(1)$ such possible functions $f_e,$ it is enough to show that for any \emph{fixed} such function $f,$ we have at most $O(n^tk^{t+1})$ edges $e$ such that $f_e=f$. To do this we implement our strategy as described above and either find there can be at most $O(n^tk^{t+1})$ edges corresponding to this recipe or we ``get stuck'' which will precisely correspond to the existence of a $(t+1, t)$-system with $N=2t+1$, which we later show is impossible.

We make these arguments precise in the following proof.
\begin{proof}[ of \Cref{lem:reduction-to-set-system}]
    Let us assume there is no $(t+1, t)$-system with $N = 2t+1.$ Let $H$ be a $(2t+1)$-uniform hypergraph with no copy of $\S{2t+1}{t}{k}.$ Our goal is to show that $H$ has $O(n^t k^{t+1})$ edges.
    
    Let $\cF$ be the family of all functions $f\colon \binom{[2t+1]}{t} \rightarrow [2t+1]$ such that $f(I) \not\in I$ for all $I \in \binom{[2t+1]}{t}.$ We will assign to each edge of our hypergraph a type corresponding to a function $f \in \cF$. In particular, if we denote by $E_f$ the set of edges of type $f$ and for all $f \in \cF$ we show that $|E_f| = O\left(n^tk^{t+1}\right)$, we are done, since $|\cF| = (t+1)^{\binom{2t+1}{t}} = O(1).$

    For any set $S$ of $t$ vertices, fix an arbitrary cover $\Phi(S)$ of $L_S$ of size at most $(k-1)(t+1),$ given by \Cref{obs:cover}.
    
    Let $e$ be any edge and let us fix an arbitrary ordering of its vertices. We view $e$ as a $(2t+1)$-tuple $(u_1, \dots, u_{2t+1}).$ We define a function $f_e \in \cF$, which will be the type of $e$, as follows. Consider any set $I \in \binom{[2t+1]}{t}$ and let $e_I:=\{u_i \: \mid \: i \in I\}$ be the subset of vertices of $e$ with indices in $I$. By definition of $\Phi$ as a cover, $\Phi(e_I) \cap e \neq \emptyset,$ so there exists an index $j \in [2t+1] \setminus I$ such that $u_j \in \Phi(e_I).$ Let $j$ be an arbitrary such index and set $f_e(I) = j.$ One should think of $f_e$ as an assignment to any subset of $t$ vertices of $e$ another vertex of $e$ which belongs to the cover $\Phi$ of this $t$-vertex subset.  
    
    For $f \in \cF$ we set $E_f = \{ e \in E(H) \; \vert \; f_e = f\}.$ Let us fix a function $f \in \cF$ for the remainder of the proof, as discussed before, it is enough to show that $|E_f| = O(n^tk^{t+1}).$ 
    
    For every set $J \in \binom{[2t+1]}{t}$ we construct an \emph{extending sequence} $J=J_0 \subseteq J_1 \subseteq \dots$ together with an auxiliary sequence $I_1, I_2, \dots$, as follows. For $i \geq 0,$ suppose there is a set $I \in \binom{J_i}{t}$ such that $f(I) \not\in J_i.$ Let $I_{i+1}$ be an arbitrary such set, let $J_{i+1} = J_i \cup \{f(I_{i+1})\}$ and note that $|J_{i+1}| = |J_i| + 1.$ If there is no such set, that is, $f(I) \in J_i$ for all $I \in \binom{J_i}{t},$ the sequence ends and we use $C(J) = J_i$ to denote the last set in the extending sequence starting with $J$.
    
    We next show that if there exists a set $J$ for which $C(J)=[2t+1],$ we have the desired bound on the number of edges of type $f$.
    \begin{claim*}
        If there exists a set $J \in \binom{[2t+1]}{t}$ such that $C(J) = [2t+1],$ then $|E_f| = O(n^t k^{t+1}).$
    \end{claim*}
    \begin{proof}
        Let $J \in \binom{[2t+1]}{t}$ such that $C(J) = [2t+1].$ Let $J=J_0 \subseteq J_1 \subseteq \ldots \subseteq J_{t+1} = [2t+1]$ be the extending sequence of $J$ and $I_1, I_2 \dots, I_{t+1}$ the auxiliary sequence, as given above. For convenience, let us reorder the coordinates so that $J_i = [t+i]$ for $0 \leq i \leq t+1$ and consequently $f(I_i) = t+i,$ for $1 \le i \le t+1.$ We will construct a set of $O(n^t k^{t+1})$ different $(2t+1)$-tuples of vertices $(v_1, \dots, v_{2t+1})$ which contains all edges in $E_f$ (also viewed as $(2t+1)$-tuples). First, we choose $v_1, \dots, v_t$ to be any $t$ distinct vertices. For $1 \leq i \leq t+1,$ let us assume we have already chosen $v_1, \dots, v_{t+i-1}$ and let $S_i = \{ v_j \, \vert \, j \in I_i \}.$ As $I_i \subseteq J_{i-1}= [t+i-1]$ and we have already chosen the vertices $v_1, \dots, v_{t+i-1},$ the set $S_i$ is well-defined. Then we choose $v_{t+i}$ to be any vertex in $\Phi(S_i).$ Let $T$ be the set of all $(2t+1)$-tuples constructed this way. One should think of this process as following the recipe given by type $f$ to generate all $(2t+1)$-tuples which could possibly correspond to an edge (viewed as a $(2t+1)$-tuple) of type $f$.
        
        First observe that there are at most $n^t$ ways to choose the vertices $v_1, \dots, v_t$ and at most $(k-1)(t+1)$ ways to choose each subsequent vertex $v_i,$ for $t+1 \leq i \leq 2t+1,$ thus $|T| \leq n^t \cdot \left((k-1)(t+1)\right)^{t+1} = O(n^t k^{t+1}).$ 
        
        Secondly, let us show that $T$ indeed contains all edges in $E_f.$ For the sake of contradiction, consider an edge $e = (u_1, \dots, u_{2t+1}) \in E_f$ and suppose that $e \not\in T.$ Then there is an index $i$ such that we have chosen $v_j = u_j,$ for $1 \leq j \leq i-1,$ but we never choose $v_i = u_i.$ Clearly $i > t,$ as any $t$-tuple of distinct vertices is chosen for $v_1, \dots, v_t.$ Note that $S_{i-t} = \{ v_j \, \vert \, j \in I_{i-t} \}=\{ u_j \, \vert \, j \in I_{i-t} \}=e_{I_{i-t}},$ since $I_{i-t} \subseteq J_{i-t-1}=[i-1].$ In particular, since we choose $v_i$ to be any member of $\Phi(S_{i-t}),$ we have $u_i \notin \Phi(S_{i-t}).$ On the other hand, by our initial relabelling $f(I_{i-t}) = i$ and since we assumed that $e$ is of type $f$, so $f_e(I_{i-t})=i$, this means that $u_i \in \Phi(e_{I_{i-t}})=\Phi(S_{i-t}),$ a contradiction.
        
        We conclude that $E_f \subseteq T$ and thus $|E_f| = O(n^t k^{t+1}),$ as claimed.
    \end{proof}
    
    Finally, we are left with the case when $C(J) \neq [2t+1]$ for all $J \in \binom{[2t+1]}{t}.$ In this case we construct a $(t+1, t)$-system on the ground set $[2t+1],$ contradicting our assumption. Let $N = 2t+1$ and define 
    \[ \cA = \left\{ C \subsetneq [N] \; \big\vert \; f(I) \in C, \forall I \in \binom{C}{t} \right\}. \]
    It is only left to verify that $\cA$ is indeed a $(t+1, t)$-system. Consider arbitrary sets $A, B \in \cA.$ For any $I \in \binom{A \cap B}{t},$ we have $f(I) \in A$ because $A \in \cA,$ and $f(I) \in B$ because $B \in \cA.$ Thus, $f(I) \in A \cap B,$ implying $A \cap B \in \cA,$ that is, $\cA$ is intersection-closed. Consider a set $J \in \binom{N}{t}.$ By assumption, $C(J) \subsetneq [2t+1]$ and satisfies $f(I) \in C(J)$ for all $I \in \binom{C(J)}{t}.$ Hence, $C(J) \in \cA$ which shows that any subset of $[N]$ of size $t$ is contained in some set in $\cA.$ Note that $N \equiv t \pmod{t+1}$ and $[N] \not\in \cA$ by definition. Consider any $J \in \binom{N}{t}.$ As $f \in \cF,$ we have $f(J) \not\in J,$ so $J \not\in \cA.$ Therefore, no set $A \in \cA$ satisfies $|A| \equiv N \pmod{t+1}$ and we conclude that $\cA$ is a $(t+1, t)$-system. 
\end{proof}

\subsection{The existence of \texorpdfstring{$(t+1,t)$}{(t+1,t)}-systems}
\label{sec:existence}
In this subsection we study the existence of $(t+1, t)$-systems. We begin by showing that no $(t+1,t)$-system exists when $t+1$ is a prime power, regardless of the value of $N$. We include our proof as it is more elementary and much simpler than the previous ones given in \cite{csp} and \cite{submodular}.

\begin{prop} \label{prop:prime-power}
    Let $t+1$ be a prime power. Then there is no $(t+1, t)$ system.
\end{prop}
\begin{proof}

      Let $t + 1 = p^\alpha$ for some prime $p$ and integer $\alpha \geq 1.$ First we calculate the residues of $\binom{a}{t}$ modulo $p$ for any value of $a.$ Let $a = \sum_{i=0}^\ell a_ip^i$ be the expansion of $a$ in base $p,$ where $\ell \geq \alpha-1$ and note that $t = p^\alpha - 1 = \sum_{i=0}^{\alpha-1} (p-1)p^i.$ By Lucas' theorem \cite{lucas1878theorie}:
      \[ \binom{a}{t} \equiv \prod_{i=0}^{\alpha-1} \binom{a_i}{p-1} \pmod{p}. \]
      The product on the right hand side is nonzero if and only if $a_i = p-1$ for all $0 \le i \le \alpha-1.$ Thus,
      \begin{equation} \label{eq:binom-residues}
          \binom{a}{t} \equiv 0 \pmod{p} \iff a \not\equiv -1 \pmod{t+1}.
      \end{equation}
      
    Suppose $\cA$ is a $(t+1, t)$-system. By definition, each $t$-subset of $[N]$ is contained in some set in $\cA.$ Hence, we can use the inclusion-exclusion principle to count the number of $t$-subsets of $[N]$ as follows:
    \[ \binom{N}{t} = \sum_{i=1}^{|\cA|} (-1)^{i+1} \sum_{\substack{\cF \subseteq \cA\\|\cF| = i}} \binom{\left| \bigcap_{F \in \cF} F \right|}{t}. \]
    By \eqref{eq:binom-residues}, the left hand side has nonzero residue modulo $p.$ Recall that $\cA$ is intersection-closed and for all $A \in \cA,$ we have $|A| \not\equiv -1 \pmod{t+1}.$ Hence, by \eqref{eq:binom-residues}, all the summands on the right hand side are congruent to $0$ modulo $p,$ a contradiction.
\end{proof}

In the rest of this subsection we show that no $(t+1, t)$-system with $N=2t+1$ exists, which combined with \Cref{lem:reduction-to-set-system} completes the proof of \Cref{thm:balanced}. As mentioned in the introduction, the restricted ground set size $N=2t+1$ plays a crucial role here, since it was shown in \cite{csp} that in general $(t+1,t)$-systems do exist whenever $t+1$ is not a prime power. 

We will use the following extremal set theoretic result of Frankl and Katona \cite{frankl1979if} and we include its proof for completeness.

\begin{prop} \label{prop:frankl-katona}
    Any collection of $m+1$ not necessarily distinct subsets of $[m]$ contains $s$ sets whose intersection has size exactly $s-1,$ for some $s, \, 1 \leq s \leq m+1.$
\end{prop}
\begin{proof}
    We prove the statement by induction on $m$, with the base case $m=0$ being trivial. Now let $m\ge 1, \,$ $A_1, \dots, A_{m+1} \subseteq [m]$ and suppose, for the sake of contradiction, that no $s$ sets have intersection of size exactly $s-1$ for any $s, \, 1 \le s \le m+1.$ For $x \in [m],$ denote by $d(x)$ the number of sets $A_1, \dots, A_{m+1}$ containing $x.$ 
    
    First we show that $x \in A_i$ implies $d(x) \le |A_i|.$ Suppose $x \in A_i$ and $d(x) > |A_i|$ for some $x \in [m], i \in [m+1].$ Consider the collection of not necessarily distinct sets $\{ (A_j \cap A_i) \setminus \{x\} \, \vert \, j \neq i, \, x \in A_j \}.$ There are $d(x) - 1 > |A_i| - 1$ of these sets and they are subsets of $A_i \setminus \{x\}$ which has size $|A_i| - 1 \le m-1.$ Thus, we may use the induction hypothesis to conclude there exist a positive integer $s$ and indices $j_1 < \ldots < j_s$ such that $\left| \bigcap_{\ell=1}^s ((A_{j_\ell} \cap A_i) \setminus \{x\}) \right| = s-1.$ But then $\left|A_i \cap \bigcap_{\ell=1}^s (A_{j_\ell}) \right| = s,$ contradicting our assumption. Therefore, $d(x) \le |A_i|,$ for any $i \in [m+1], \, x \in A_i.$
    
    By the induction hypothesis, the union of any $k$ sets must have size at least $k,$ for any $1 \le k \le m.$ By Hall's theorem, there is a system of distinct representatives for $A_1, \dots, A_m$ so upon relabelling we may assume that $i \in A_i,$ for all $i \in [m].$ In particular, we must have $d(i) \le |A_i|,$ for all $i \in [m]$. We now have
    \[ \sum_{i=1}^{m+1} |A_i| = \sum_{i=1}^m d(i) \le \sum_{i=1}^m |A_i|. \]
    Hence, $A_{m+1} = \emptyset,$ contradicting our assumption for $s=1$.
\end{proof}

We are now ready to complete the proof of the balanced case.
\begin{lem} \label{lem:no-system}
    There exists no $(t+1, t)$-system on a ground set of size $N = 2t+1.$
\end{lem}

\begin{proof}
    Suppose $\cF$ is a $(t+1, t)$-system on $[N] = [2t+1].$ Let $S$ be a minimal subset of $[N]$ not covered by $\cF,$ i.e. such that $S \not\subseteq F, \, \forall F \in \cF.$ As $[N] \not\in \cF,$ the set $S$ is well-defined. Since we know that any $t$-element subset is covered by the second defining property of the $(t+1,t)$-system, we must have $\ell = |S| \ge t+1$. Without loss of generality, assume $S = \{1, \dots, \ell\}.$ By minimality of $S,$ for any $i\in [\ell]$ there must exist a set in $\cF$ which covers $S \setminus \{i\}.$ In particular, let $A_1, \dots, A_{t+1} \in \cF$ be sets such that $A_i \cap S = S \setminus \{i\},$ for $i\in [t+1].$ 
    
    Our plan is to find a collection of sets among $A_1, \dots, A_{t+1}$ whose intersection has size precisely $t,$ contradicting the third property of a $(t+1, t)$-system. Observe that for any $m$ of these sets, the size of their intersection inside $[t+1]$ is exactly $t+1 - m.$ To control the size of this intersection outside $[t+1],$ we will use \Cref{prop:frankl-katona}.
    
    For $1 \le i \le t+1,$ let $D_i = A_i \setminus [t+1].$ The sets $D_1, \dots, D_{t+1}$ all belong to the set $\{t+2,\ldots, 2t+1\}$ of size $t,$ so by \Cref{prop:frankl-katona}, there exists an integer $s$ and indices $1 \le j_1 < j_2 < \ldots < j_s \le t+1$ such that $\left|\bigcap_{i=1}^s D_{j_i}\right| = s-1.$ Let $A \coloneqq  \big(\bigcap_{i=1}^s A_{j_i}\big) \in \cF.$ We have
    \[ |A| = |A \cap [t+1]| + |A \setminus [t+1]| = |[t+1] \setminus \{j_1, \dots, j_s\}| + \big|\bigcap_{j=1}^s D_{j_i}\big|  = t+1-s + s - 1 = t. \]
    Thus, $|A| \equiv N \pmod{t+1},$ contradicting the assumption that $\cF$ is a $(t+1, t)$-system.
\end{proof}

\section{Reduction to the balanced case}
\label{sec:middle-case-reduction}
In this section we will complete the proof of the upper bound part of \Cref{thm:main}. We begin by proving an auxiliary lemma.

\begin{lem}\label{lm:many-forbidden-stars}
Let $r \ge 2$ and $r>\ell\ge 1$. Any $n$-vertex $r$-uniform hypergraph $H$ without $\S{r}{s}{k}$ for all $\ell \le s \le r-1$ has at most $O(n^{\ell}k^{r-\ell})$ edges.
\end{lem}
\begin{proof}
We claim that for any subset $S$ of at least $\ell$ and at most $r$ vertices of $H,$ its link graph $L_S$ has size at most $(rk)^{r-|S|}$. We prove this by reverse induction on $|S|$. As the basis if $|S|=r,$ its link graph consists of only the empty set and in particular has size $1$. Now assume that for all subsets of size $|S|+1$ our claim holds. The fact there is no $\S{r}{|S|}{k}$ in $H$ implies through \Cref{obs:cover} that there exists a cover of $L_S$ of size at most $(r-|S|)(k-1)\le rk$. In particular, the size of $L_S$ is at most the sum of the sizes of the link graphs of $L_{S \cup \{v\}}$ as $v$ ranges over the vertices in the cover. Using the inductive assumption to bound the size of each of these link graphs we conclude $|L_{S}| \le rk \cdot (rk)^{r-|S|-1}=(rk)^{r-|S|},$ as desired.

Since there are at most $n^\ell$ sets of size $\ell$ we conclude there are at most $n^\ell (rk)^{r-\ell} \le O(n^\ell k^{r-\ell})$ edges in $H$.
\end{proof}

We are now ready to complete the proof of the upper bound of \Cref{thm:main} which we restate below for convenience.

\begin{thm}\label{thm:main-ub} Let $r > t \ge 0,$ then
$$
\ex(n,\S{r}{t}{k})\le 
\begin{cases}
O(n^{r-t-1}k^{t+1}) & \text{ if } t \le \frac{r-1}{2}, \\
O(n^{t}k^{r-t}) & \text{ if } t \ge \frac{r-1}{2}.
\end{cases}
$$
\end{thm}

\begin{proof}
We proceed by induction on $r$. The base case of $r=1$ is immediate. Let us now assume $r \ge 2$ and we assume the theorem holds for all sunflowers with smaller uniformity.

Let $H$ be an $n$-vertex $r$-uniform hypergraph with at least $Cn^{r-t-1}k^{t+1}$ edges if $2t+1 \le r$ and at least $Cn^tk^{r-t}$ edges if $2t+1 \ge r$, where we choose $C=C(r)$ large enough.

If $r \le 2t$ then there is a vertex $v$ of $H$ belonging to at least $rC\cdot n^{t-1}k^{r-t}$ edges. Looking at the link graph of $v$, it is $(r-1)$-uniform, has $n-1$ vertices and at least $rC\cdot n^{t-1}k^{r-t}$ edges. Since $t-1 \ge \frac{r-2}{2}$ the induction hypothesis tells us $\ex(n-1,\S{r-1}{t-1}{k}) \le O(n^{t-1}k^{r-t})$. Choosing $C$ large enough this guarantees a copy of $\S{r-1}{t-1}{k}$ inside the link graph of $v$, which together with $v$ makes the desired copy of $\S{r}{t}{k}$ in $H$, as desired.

If $r=2t+1$ the result follows from \Cref{thm:balanced} so we are left with the case when $r \ge 2t+2$. Call a set $L$ of vertices of $H$ \emph{$m$-extending} if there are $m$ edges of $H$, all of which contain $L$ and are all disjoint outside of $L$. In other words, they make a sunflower with $L$ as a kernel.

The following claim exploits the fact that one can use $m$-extending sets to build a copy of $\S{r}{t}{k}$ in $H.$

\begin{claim} \label{clm:few-extending-sets}
    For any $1 \le i \le r-t-1,$ there are at most $\ex(n, \S{r-i}{t}{k})$ $rk$-extending sets of size $r-i$. 
\end{claim}
\begin{proof}
    Let us assume the contrary, so there exist $k$ distinct $rk$-extending sets of $r-i$ vertices forming a copy of $\S{r-i}{t}{k}.$ Then we can greedily extend them, one by one, to obtain a copy of $\S{r}{t}{k}$ in $H.$ Indeed, at any point, the union of the sets we constructed has size at most $(k-1)r.$ Hence, the current one, having $rk$ vertex disjoint choices of an $i$ vertex set for its extension, can be chosen disjointly from all the previous ones. Therefore, $H$ contains a copy of $\S{r}{t}{k},$ a contradiction.
\end{proof}

We now show that by removing any edge of $H$ containing an $rk$-extending $r-i$ vertex set for some $i \ge 1$ we remove at most half of the edges from $H$. 

\begin{claim}
There are at most $O(n^{r-t-1}k^{t+1})$ edges of $H$ which contain an $rk$-extending $r-i$ vertex set for some $i \le r-t-1$.
\end{claim}
\begin{proof}
We will first show there are at most $O(n^{r-t-1}k^{t+1})$ edges which contain an $rk$-extending $r-1$ vertex set. Since $t \le \frac{r-2}{2},$ we know that there are at most $\ex(n, \S{r-1}{t}{k}) \le O(n^ {r-t-2}k^{t+1})$ $rk$-extending $r-1$ vertex sets by \Cref{clm:few-extending-sets}. Each of them belongs to at most $n$ different edges so indeed at most $O(n^{r-t-1}k^{t+1})$ edges of $H$ contain such a set. Let us remove all these edges from $H$ to create $H_2$. We now assume that we are given a hypergraph $H_{i}$ which does not contain an $rk$-extending $r-j$ vertex set for any $j < i$ and we wish to count how many edges of $H_i$ contain an $rk$-extending $r-i$ vertex set. 

By \Cref{clm:few-extending-sets}, there are at most $\ex(n, \S{r-i}{t}{k})$ such $r-i$ sets. Given a fixed one, observe that its link graph in $H_i$, being an $i$-uniform hypergraph, contains no $\S{i}{j}{rk}$ for any $1\le j < i$. Indeed, adding back in the $r-i$ set this would give rise to an $rk$-extendable set of size $r-i+j$ in $H_i$, which we have removed in one of the previous steps. By \Cref{lm:many-forbidden-stars} (with $\ell=1$ and $r=i$) this tells us the link graph can have size at most $O(nk^{i-1})$. Combining, this implies that at most  $O(nk^{i-1}) \cdot \ex(n, \S{r-i}{t}{k})$ edges of $H_i$ contain an $rk$-extending $r-i$ vertex set.

If $2t \le r-i-1$ then $$\ex(n, \S{r-i}{t}{k}) \le O(n^{r-i-t-1}k^{t+1})$$ so the total number of our edges is $$O(n^{r-i-t}k^{t+i})\le O(n^{r-t-1}k^{t+1})$$ since $i \ge 1$ and we may assume $k \le n$ (otherwise the claim is vacuous). 

If $2t > r-i-1,$ then $$\ex(n, \S{r-i}{t}{k}) \le O(n^{t}k^{r-i-t}),$$ so the total number of our edges is $$O(n^{t+1}k^{r-t-1})\le O(n^{r-t-1}k^{t+1})$$
which holds since $r\ge 2t+2$ and since, as before, we may assume $k\le n$.
\end{proof}

After removing any of the edges counted by the claim, the remainder graph has no $\S{r}{r-i}{rk}$ for $1 \le i < r-t$, since a kernel of any such sunflower would be an $rk$-extending set of $r-i$ vertices. \Cref{lm:many-forbidden-stars} with $\ell=t+1$ now tells us that there at most $O(n^{t+1}k^{r-t-1})\le O(n^{r-t-1}k^{t+1})$ edges remaining. Since we also removed at most this many edges by the claim, we are done. 
\end{proof}

\section{Lower bounds}

In this section we will establish the lower bounds on $\ex(n,\S{r}{t}{k})$ which will complete the proof of \Cref{thm:main}. Observe first that if $k=\Omega(n)$ then our goal is to show that $\ex(n,\S{r}{t}{k}) \ge \Omega(n^r)$. This holds since the complete $r$-uniform hypergraph on $t+k(r-t)-1=|\S{r}{t}{k}|-1$ vertices has $\Omega(k^r)=\Omega(n^r)$ edges and clearly does not contain $\S{r}{t}{k}.$ So we only need to prove our lower bounds when $k \le cn$ for an arbitrary constant $c$.

We will use two fundamentally different classes of examples depending on which side of the balanced case we are on, in other words whether the size of the kernel of our forbidden sunflower is larger or smaller than half of the uniformity. We begin with the easier case when the kernel is smaller.

\begin{lem}\label{lm:lower-bound-small-centre}
Let $r\ge 2t+1 \ge 1$ and $k \le n/2,$ then
$$\ex(n,\S{r}{t}{k}) \ge \Omega(n^{r-t-1}k^{t+1}).$$
\end{lem}

\begin{proof}

Let $s=\max\{k-1,t+1\}$. We take $A$ to be a set of $n-s$ vertices and $B$ a set of $s$ vertices. We now choose $H$ to be the hypergraph with vertex set $A
\cup B$ consisting of all edges which have exactly $r-t-1$ vertices in $A$ and $t+1$ vertices in $B$. Such an $H$ has $$\binom{n-s}{r-t-1} \cdot \binom{s}{t+1} \ge \Omega(n^{r-t-1}k^{t+1})$$
edges, where the case when $s=t+1$ follows since this means $k \le t+2$ is a constant.
To complete the proof we will show it does not contain $\S{r}{t}{k}$. If $s = t+1,$ any two edges intersect in more than $t$ vertices so the statement is trivial. Otherwise, $s=k-1$ and let us assume to the contrary that there is a copy $S$ of $\S{r}{t}{k}$ in $H$. Note that each edge of $S$ must have one of its $r-t$ vertices not in the kernel of $S$ belonging to $B$ (since every edge of $H$ has exactly $r-t-1$ vertices inside $A$). If we take one such vertex for each of the $k$ edges of $S$ we know they all belong to $B$ and must be distinct since they do not belong to the kernel of $S$. This is a contradiction to $|B|=k-1$, completing the proof.
\end{proof}

A slightly refined version of this construction where one takes the set $B$ to be larger but only allows certain subsets of size $t+1$ of $B$ to participate in the edges is conjectured by Frankl and F\"uredi \cite{FF87} to be nearly optimal for this case, when both $k$ and $r$ are assumed to be constant. In particular, they conjecture that one should choose the allowed subsets of $B$ based on the optimal example for the classical Erd\H{o}s-Rado sunflower conjecture mentioned in the introduction, so that there is no sunflower of size $k,$ for any size of the kernel, made by the allowed $t+1$-vertex subsets of $B$. 

Turning to the other, slightly more involved regime, let us first describe the example which Frankl and F\"uredi conjectured is optimal in this case as it will provide some motivation for our actual construction. Let us start with a $(t,(r-t)k+t-1,n)$ Steiner system $S$, defined as an $((r-t)k+t-1)$-uniform $n$-vertex hypergraph in which every $t$ element subset of the vertices belongs to precisely one edge. That such objects exist, for fixed values of $t,r,k$, and infinitely many values of $n$ follows from a recent remarkable breakthrough result of Keevash \cite{designs}, although for our purposes a much older result of R\"odl \cite{Rodl85} which says that such objects in which only $1-o(1)$ proportion of $t$-element subsets belong to an edge would be sufficient. To get the example of Frankl and F\"uredi, one now takes $H$ to be the $r$-uniform hypergraph obtained by choosing for every edge of $S$ all its $r$ element subsets to be edges of $H$. It is not hard to check that such an $H$ has no $\S{r}{t}{k}$ and that it indeed has $\Omega(n^{t}k^{r-t})$ edges. This example establishes the lower bound in our problem as well provided $k$ is also a constant, however once one allows $k$ to grow with $n$ the above approach stops to work and, in particular, the required Steiner systems can be seen not to even exist if $k$ is large enough in terms of $n$. The approximate approach of \cite{Rodl85} allows $k$ to grow with $n$ but only very slowly. Our new contribution here is to identify in some sense the correct, easier to ensure property, even when $k$ is linear in $n$, which still suffices to prevent the existence of $\S{r}{t}{k}$.

\begin{lem}\label{lm:lower-bound-big-centre}
Let $2t \ge r > t \ge 0$ and $k \le n/r^2,$ then
$$\ex(n,\S{r}{t}{k}) \ge \Omega(n^{t}k^{r-t}).$$
\end{lem}

\begin{proof}
We will choose the edges of our desired hypergraph on a vertex set $V$ of size $n$ in multiple stages. At stage $i$ we first choose a subset of at most $s=(r-t)k+t-1$ vertices $V_i$ and then a collection of edges $S_i$ only using vertices from $V_i$ and add them to the edge set of $H$. We will do so in a way that ensures that no two edges from different stages intersect in $t$ or more vertices. This ensures that any copy of $\S{r}{t}{k}$ we may find must be using edges of only a single $S_i$, which is impossible since the number of vertices in a copy of $\S{r}{t}{k}$ is $(r-t)k+t>|V_i|$. 

Let us assume we have already chosen $V_1,\ldots, V_i$ and $S_1,\ldots, S_i$ for some $0 \le i \le \Omega((n/k)^t)$.
Let us choose $v_1,\ldots, v_{s}$ independently and uniformly at random from $V$. Observe first that the probability that $v_1,\ldots,v_r$ are all distinct is at least $1-\binom{r}{2}\cdot \frac{1}{n} \ge \frac{1}{2}$. Second, conditioning on $v_1, \dots, v_r$ being all distinct, observe that the probability that $t$ vertices among them were chosen inside of some $V_j$ with $j\le i$ is by a union bound at most 
$$i \cdot \binom{r}{t} \cdot \left(\frac{(r-t)k+t-1}{n}\right)^t \le \frac{1}{2}.$$
This means that with probability at least $1/4$ the vertices $v_1,\ldots, v_r$ are all distinct and at most $t-1$ of them belong to the same $V_j,$ for any $j\le i$. The same holds for any $r$-tuple of our random vertices of which there are $\binom{(r-t)k+t-1}{r}$. In particular, there is an outcome in which at least $\frac{1}{4}\cdot \binom{(r-t)k+t-1}{r} \ge \Omega(k^r)$ of the $r$-tuples satisfy these two properties. We set $V_{i+1}=\{v_1,\ldots, v_s\}$ and we let $S_{i+1}$ consist of all such $r$-tuples. In particular, $|S_{i+1}| \ge \Omega(k^r)$ and no edge in $S_{i+1}$ contains more than $t-1$ vertices in any $V_{j}$ with $j\le i$. The latter condition ensures, as we described above, that there is no $\S{r}{t}{k},$ and the former combined with the fact we may perform $\Omega((n/k)^t)$ many stages that we in total construct at least $\Omega(n^tk^{r-t})$ edges, as desired. 
\end{proof}

\section{Concluding remarks}
\vspace{-0.35cm}

In this paper we determined the correct order of magnitude of the Tur\'an number of sunflowers of any uniformity, in terms of both the size of the underlying graph and the size of the sunflower.
In the process we link this problem to a number of interesting topics from a variety of areas including  optimisation theory, theoretical computer science and design theory.

Frankl and F\"uredi determined the correct dependence of $\ex(n,\S{r}{t}{k})$ on $n$, our result generalises this and provides the correct dependence on both $n$ and $k$ so a natural next question would be to, in addition, determine the correct dependence on $r$. Here the first step which already might hold some interest is to determine what the answer is in the simplest case $t=1$, where for example we know the answer up to lower order terms for $r=3$, thanks to Chung \cite{chung1983unavoidableS}. 

As mentioned in the introduction one of the motivations of Chung and Erd\H{o}s (see e.g.\ \cite{chung1983unavoidable,chung1987unavoidable}) for raising the question of what happens if one allows $k$ to grow with $n$ is the relevance of this problem to the so-called unavoidability problem. It asks, across all $r$-uniform hypergraphs $H$ with a fixed number of edges what is the smallest $n$-vertex Tur\'an number of $H$. The answer is now known, up to a constant factor, in the graph case as well as in the $3$ and $4$ uniform cases. In fact Chung and Erd\H{o}s say that one of the main obstructions to the solution of this problem lies in the fact that $\ex(n,\S{r}{t}{k})$ is not that well understood. The reason behind this is the fact that in all known cases sunflowers, and in particular ones with size growing with $n$, provide the basic building blocks for the actual optimal unavoidable hypergraphs. As such, our result completely resolves this issue and makes an essential first step towards a resolution of the unavoidability problem for any uniformity.

Sunflowers are perhaps the most natural examples of sparse hypergraphs. This leads to the question of whether results similar to ours can be obtained for other natural families of sparse hypergraphs. In particular, a specific family called generalised stars was suggested in \cite{unavoidability} based on the fact that the actual optimal unavoidable hypergraphs, at least in the currently known cases, always turn out to be generalised stars. The $r$-uniform generalised star $\mathcal{S}t_r(d_1,\ldots,d_{r})$ is defined recursively to consist of $d_1$ vertex disjoint copies of the hypergraph consisting of a ``root'' vertex whose link graph is equal to $\mathcal{S}t_{r-1}(d_2,\ldots,d_{r})$. So for example $\S{r}{t}{k}=\mathcal{S}t_r(1,\ldots,1,k,1,\ldots, 1)$ where $k$ appears as the $t$-th entry. Apart from being interesting in its own right, understanding the $n$-vertex Tur\'an number of generalised stars, where $d_i$ is allowed to grow with $n$ seems to be the next essential step towards the proof of the unavoidability conjecture of Chung and Erd\H{o}s.
 

Finally, another potentially interesting further direction might be to understand what happens if we exclude multiple sunflowers simultaneously. In particular, what is the maximum number of edges in an $r$-uniform hypergraph which contains no $\S{r}{t}{k}$ for any $t$ in some given subset $T \subseteq [r-1]$. Our result resolves the case when $T$ is a single element subset. Some further examples include taking $T=\{0,1,\ldots,r-1\}$ in which case we find ourselves in the setting of the classical Erd\H{o}s-Rado sunflower conjecture discussed in the introduction.
In fact, understanding the answer to this question for $T=\{0,1\ldots, t-1\}$ has been raised as one of the initial problems in the proposal by Kalai for the Polymath project \cite{Polymath} surrounding this conjecture. 
Another natural instance is when $T=\{\ell,\ell+1,\ldots, r-1\}$, which in some sense translates to $\ell$-element subsets having a bounded codegree. We note that in all these examples determining the correct dependency on $k$ is fairly easy and the main question is the correct dependency on $r$. However, as evidenced by our result on the single element set case, even the correct dependency on $k$ for an arbitrary set $T$ seems to be far from immediate.

\providecommand{\bysame}{\leavevmode\hbox to3em{\hrulefill}\thinspace}
\providecommand{\MR}[1]{}
\providecommand{\MRhref}[2]{%
  \href{http://www.ams.org/mathscinet-getitem?mr=#1}{#2}
}
\providecommand{\href}[2]{#2}

\end{document}